\documentclass{article}

\usepackage{amsthm,amsmath,amssymb}
\usepackage{tikz,url}
\usepackage{mathpazo}
\usepackage{fancybox}
\usepackage{tikz}
\usepackage{graphicx} 
\usepackage[normalem]{ulem}
\usepackage{hyperref}

\usetikzlibrary{matrix}
\usetikzlibrary{calc}
\usetikzlibrary{positioning}

\newcounter{row2}
\newcounter{row5}
\newcounter{row6}
\newcounter{row7}
\newcounter{col2}
\newcounter{col7}
\newcounter{col6}
\newcounter{col5}

\newcommand\setrowFIVE[5]{
  \setcounter{col5}{1}
  \foreach \n in {#1, #2, #3, #4, #5} {
    \edef\x{\value{col5} - 0.5}
    \edef\y{5.5 - \value{row5}}
    \node[anchor=center] at (\x, \y) {\n};
    \stepcounter{col5}
  }
  \stepcounter{row5}
}
\newcommand\setrowFIVEcolored[5]{%
  \setcounter{col5}{1}%
  \foreach \tc/\txt in {#1, #2, #3, #4, #5}{%
    \edef\x{\value{col5} - 0.5}%
    \edef\y{5.5 - \value{row5}}%
    \node[anchor=center, text=\tc] at (\x, \y) {\txt};%
    \stepcounter{col5}%
  }%
  \stepcounter{row5}%
}

\newcommand\setrowSIX[6]{
  \setcounter{col6}{1}
  \foreach \n in {#1, #2, #3, #4, #5, #6} {
    \edef\x{\value{col6} - 0.5}
    \edef\y{6.5 - \value{row6}}
    \node[anchor=center] at (\x, \y) {\n};
    \stepcounter{col6}
  }
  \stepcounter{row6}
}
\newcommand\setrowSIXcolored[6]{%
  \setcounter{col6}{1}%
  \foreach \tc/\txt in {#1, #2, #3, #4, #5, #6}{%
    \edef\x{\value{col6} - 0.5}%
    \edef\y{6.5 - \value{row6}}%
    \node[anchor=center, text=\tc] at (\x, \y) {\txt};%
    \stepcounter{col6}%
  }%
  \stepcounter{row6}%
}

\newcommand\setrowSEVEN[7]{
  \setcounter{col7}{1}
  \foreach \n in {#1, #2, #3, #4, #5, #6, #7} {
    \edef\x{\value{col7} - 0.5}
    \edef\y{7.5 - \value{row7}}
    \node[anchor=center] at (\x, \y) {\n};
    \stepcounter{col7}
  }
  \stepcounter{row7}
}
\newcommand\setrowSEVENcolored[7]{%
  \setcounter{col7}{1}%
  \foreach \tc/\txt in {#1, #2, #3, #4, #5, #6, #7}{%
    \edef\x{\value{col7} - 0.5}%
    \edef\y{7.5 - \value{row7}}%
    \node[anchor=center, text=\tc] at (\x, \y) {\txt};%
    \stepcounter{col7}%
  }%
  \stepcounter{row7}%
}

\newcommand{\circl}{\hspace{4.3mm}\vspace{3mm}\circle{12}}
\newtheorem{theorem}{Theorem}
\newtheorem{lemma}{Lemma}
\newtheorem{corollary}{Corollary}
\colorlet{darkyellow}{yellow!70!brown}

\begin{document}
\title{A Closed Form for the Pulsar Sequence}
\author{Ryan Liu$^1$ \& Vadim Ponomarenko$^2$}
\date{$^1$Del Norte High School, \url{ryanzeyuliu@gmail.com}\\ $^2$San Diego State University, \url{vponomarenko@sdsu.edu}\\[2ex] \today }

\maketitle

\section*{Introduction}

On July 17, 2025, the popular sudoku channel ``Cracking The Cryptic'' released a video \cite{ctc} featuring a puzzle by Pulsar based on a Latin square. A Latin square of size $N$ is a $N\times N$ grid filled in such that each row and column contains the numbers $1$---$N$. The puzzle introduced in the video was to complete a $9\times 9$ Latin square grid such that every row and every column contained each digit $1$---$9$. The restriction is that a fixed spiral of cells were circled where each digit in a circled cell must equal to the number of circled cells that contain that digit. We call these types of puzzles Pulsar puzzles. 
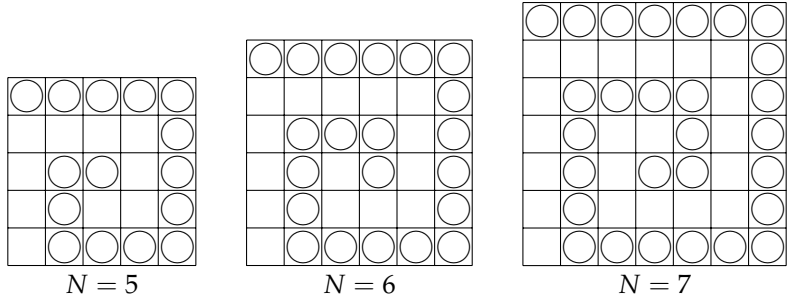
\begin{center}
\begin{tikzpicture}[scale=.5]
  \begin{scope}
    \draw (0, 0) grid (5, 5);

    \setcounter{row5}{1}
    \setrowFIVE {\circl }{ \circl}{\circl }  {\circl }{\circl}
    \setrowFIVE {}{}{}{}{\circl}
    \setrowFIVE {}{\circl}{\circl}{}{\circl}
    \setrowFIVE {}{\circl}{}{}{\circl}
    \setrowFIVE {}{\circl}{\circl}{\circl}{\circl}
    \node[anchor=center] at (2.5, -0.5) {$N=5$};
  \end{scope} 
    
\end{tikzpicture} \begin{tikzpicture}[scale=.5]

  \begin{scope}
    \draw (0, 0) grid (6, 6);

    \setcounter{row6}{1}
    \setrowSIX {\circl }{ \circl}{\circl }  {\circl }{\circl}{\circl}
    \setrowSIX {}{}{}{}{}{\circl}
    \setrowSIX {}{\circl}{\circl}{\circl}{}{\circl}
    \setrowSIX {}{\circl}{}{\circl}{}{\circl}
    \setrowSIX {}{\circl}{}{}{}{\circl}
    \setrowSIX {}{\circl}{\circl}{\circl}{\circl}{\circl}

    \node[anchor=center] at (3, -0.5) {$N=6$};
  \end{scope}
\end{tikzpicture} \begin{tikzpicture}[scale=.5]

  \begin{scope}
    \draw (0, 0) grid (7, 7);

    \setcounter{row7}{1}
    \setrowSEVEN{\circl}{\circl}{\circl }  {\circl }{\circl}{\circl}{\circl}
    \setrowSEVEN{}{}{}{}{}{}{\circl}
    \setrowSEVEN{}{\circl}{\circl}{\circl}{\circl}{}{\circl}
    \setrowSEVEN{}{\circl}{}{}{\circl}{}\circl{}
    \setrowSEVEN{}{\circl}{}{\circl}{\circl}{}{\circl}
    \setrowSEVEN{}{\circl}{}{}{}{}{\circl}
    \setrowSEVEN{}{\circl}{\circl}{\circl}{\circl}{\circl}{\circl}

    \node[anchor=center] at (3.5, -0.5) {$N=7$};
  \end{scope}
\end{tikzpicture}\\
\end{center}
\noindent Below, we have the solutions for puzzles of size $N=5,6,7$.  Note that the red numbers for all three, spiraling out from the center, agree.  This is no coincidence.\\

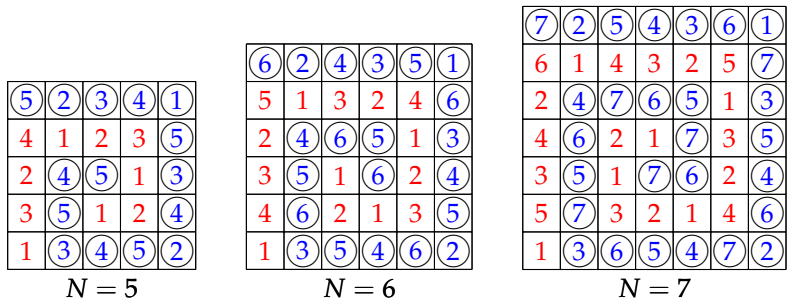
\begin{tikzpicture}[scale=.5]
  \begin{scope}
    \draw (0, 0) grid (5, 5);
    \setcounter{row5}{1}
    \setrowFIVE {\circl }{ \circl}{\circl }  {\circl }{\circl}
    \setrowFIVE {}{}{}{}{\circl}
    \setrowFIVE {}{\circl}{\circl}{}{\circl}
    \setrowFIVE {}{\circl}{}{}{\circl}
    \setrowFIVE {}{\circl}{\circl}{\circl}{\circl}
    \node[anchor=center] at (2.5, -0.5) {$N=5$};
  \end{scope} 
   \begin{scope}
    \draw (0, 0) grid (5, 5);
    \setcounter{row5}{1}
    \setrowFIVEcolored {blue/5}{blue/2}{blue/3}{blue/4}{blue/1}
    \setrowFIVEcolored {red/4}{red/1}{red/2}{red/3}{blue/5}
    \setrowFIVEcolored {red/2}{blue/4}{blue/5}{red/1}{blue/3}
    \setrowFIVEcolored {red/3}{blue/5}{red/1}{red/2}{blue/4}
    \setrowFIVEcolored {red/1}{blue/3}{blue/4}{blue/5}{blue/2}
    \node[anchor=center] at (2.5, -0.5) {$N=5$};
  \end{scope}    
\end{tikzpicture} \begin{tikzpicture}[scale=.5]

  \begin{scope}
    \draw (0, 0) grid (6, 6);
    \setcounter{row6}{1}
    \setrowSIX {\circl }{ \circl}{\circl }  {\circl }{\circl}{\circl}
    \setrowSIX {}{}{}{}{}{\circl}
    \setrowSIX {}{\circl}{\circl}{\circl}{}{\circl}
    \setrowSIX {}{\circl}{}{\circl}{}{\circl}
    \setrowSIX {}{\circl}{}{}{}{\circl}
    \setrowSIX {}{\circl}{\circl}{\circl}{\circl}{\circl}
    \node[anchor=center] at (3, -0.5) {$N=6$};
  \end{scope}
  \begin{scope}
    \draw (0, 0) grid (6, 6);
    \setcounter{row6}{1}
    \setrowSIXcolored {blue/6}{blue/2}{blue/4}{blue/3}{blue/5}{blue/1}
    \setrowSIXcolored {red/5}{red/1}{red/3}{red/2}{red/4}{blue/6}
    \setrowSIXcolored {red/2}{blue/4}{blue/6}{blue/5}{red/1}{blue/3}
    \setrowSIXcolored {red/3}{blue/5}{red/1}{blue/6}{red/2}{blue/4}
    \setrowSIXcolored {red/4}{blue/6}{red/2}{red/1}{red/3}{blue/5}
    \setrowSIXcolored {red/1}{blue/3}{blue/5}{blue/4}{blue/6}{blue/2}
    \node[anchor=center] at (3, -0.5) {$N=6$};
  \end{scope}  
\end{tikzpicture} \begin{tikzpicture}[scale=.5]

  \begin{scope}
    \draw (0, 0) grid (7, 7);
    \setcounter{row7}{1}
    \setrowSEVEN{\circl}{\circl}{\circl }  {\circl }{\circl}{\circl}{\circl}
    \setrowSEVEN{}{}{}{}{}{}{\circl}
    \setrowSEVEN{}{\circl}{\circl}{\circl}{\circl}{}{\circl}
    \setrowSEVEN{}{\circl}{}{}{\circl}{}\circl{}
    \setrowSEVEN{}{\circl}{}{\circl}{\circl}{}{\circl}
    \setrowSEVEN{}{\circl}{}{}{}{}{\circl}
    \setrowSEVEN{}{\circl}{\circl}{\circl}{\circl}{\circl}{\circl}
    \node[anchor=center] at (3.5, -0.5) {$N=7$};
  \end{scope}
  \begin{scope}
    \draw (0, 0) grid (7, 7);
    \setcounter{row7}{1}
    \setrowSEVENcolored{blue/7}{blue/2}{blue/5}{blue/4}{blue/3}{blue/6}{blue/1}
    \setrowSEVENcolored{red/6}{red/1}{red/4}{red/3}{red/2}{red/5}{blue/7}
    \setrowSEVENcolored{red/2}{blue/4}{blue/7}{blue/6}{blue/5}{red/1}{blue/3}
    \setrowSEVENcolored{red/4}{blue/6}{red/2}{red/1}{blue/7}{red/3}{blue/5}
    \setrowSEVENcolored{red/3}{blue/5}{red/1}{blue/7}{blue/6}{red/2}{blue/4}
    \setrowSEVENcolored{red/5}{blue/7}{red/3}{red/2}{red/1}{red/4}{blue/6}
    \setrowSEVENcolored{red/1}{blue/3}{blue/6}{blue/5}{blue/4}{blue/7}{blue/2}
    \node[anchor=center] at (3.5, -0.5) {$N=7$};
  \end{scope}  
\end{tikzpicture}

Previous work \cite{pulsar1} has found that not only can the $9\times9$ Pulsar puzzle be solved with a unique solution, in fact any Pulsar puzzle of size $N$ can be solved with a unique solution. Furthermore, the puzzle itself is made out of two interlocked spirals of cells. We call the sequence formed by the spiral of circled cells (starting from the center) the Dual Sequence, and we call the sequence formed by the spiral of noncircled cells (starting from the center) the Pulsar Sequence, which was proved to always have the same order regardless of the size of the puzzle. This means that the Pulsar Sequence solution for a puzzle of size $N-1$ will become initial terms for the Pulsar Sequence solution for a puzzle of size $N$.

 The first few terms of the Pulsar Sequence are represented below. 

\[1,2,1,3,2,1,4,2,3,1,5,2,3,4,1,6,2,4,3,5,1,7,2,5,4,3,6,1,8,2,6,4,5,3,\ldots\]

The paper \cite{pulsar1} defines how each spiral can further be broken up into its own pieces, where a piece is defined as consecutive cells in a single row or column. Specifically for a puzzle of size $N$, the $N$-th Dual piece will be the first row of the puzzle. The ($N-1$)-th Dual piece will be the entire $N$-th column of the grid, apart from the top square. The ($N-2$)-th Dual piece will be the entire $N$-th row of the grid, apart from the first and last square, and so on. Similarly, the ($N-1$)-th Pulsar piece will be the first column of the grid, excluding the square in the top row. The ($N-2$)-th Pulsar piece will be the entire 2nd row of the grid, apart from the square in the first and last columns. The ($N-3$)-th pulsar piece will be the entire ($N-1$)-th column of the grid, apart from the squares in the first, second, and last row, and so on. Note that firstly, in a puzzle of size $N$, there will be a total of $N$ Dual pieces and $N-1$ Pulsar pieces due to the construction of the puzzle. Then, each piece of size $N$ will contain a total of $N$ different cells. In the examples below, we mark the Dual pieces with blue lines and the Pulsar pieces with red lines.

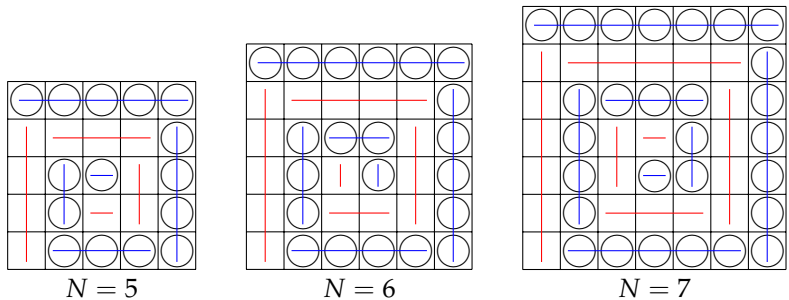
\begin{tikzpicture}[scale=.5]
  \begin{scope}
    \draw (0, 0) grid (5, 5);

    \setcounter{row5}{1}
    \setrowFIVE {\circl }{ \circl}{\circl }  {\circl }{\circl}
    \setrowFIVE {}{}{}{}{\circl}
    \setrowFIVE {}{\circl}{\circl}{}{\circl}
    \setrowFIVE {}{\circl}{}{}{\circl}
    \setrowFIVE {}{\circl}{\circl}{\circl}{\circl}
    \node[anchor=center] at (2.5, -0.5) {$N=5$};
      \draw [blue](0.3,4.5) -- (4.8,4.5);
      \draw [blue](4.5,3.8) -- (4.5,0.2);
      \draw [blue](1.2,0.5) -- (3.8,0.5);
      \draw [blue](1.5,1.2) -- (1.5,2.8);
      \draw [blue](2.2,2.5) -- (2.8,2.5);
      \draw [red](0.5,0.2) -- (0.5,3.8);
      \draw [red](1.2,3.5)--(3.8,3.5);
      \draw [red](3.5,2.8)--(3.5,1.2);
      \draw [red](2.2,1.5)--(2.8,1.5);
    
  \end{scope}
\end{tikzpicture} \begin{tikzpicture}[scale=.5]

  \begin{scope}
    \draw (0, 0) grid (6, 6);

    \setcounter{row6}{1}
    \setrowSIX {\circl }{ \circl}{\circl }  {\circl }{\circl}{\circl}
    \setrowSIX {}{}{}{}{}{\circl}
    \setrowSIX {}{\circl}{\circl}{\circl}{}{\circl}
    \setrowSIX {}{\circl}{}{\circl}{}{\circl}
    \setrowSIX {}{\circl}{}{}{}{\circl}
    \setrowSIX {}{\circl}{\circl}{\circl}{\circl}{\circl}

    \node[anchor=center] at (3, -0.5) {$N=6$};
     \draw [blue](0.3,5.5) -- (5.8,5.5);
          \draw [blue](5.5,4.8) -- (5.5,0.2);
          \draw [blue](1.2,0.5) -- (4.8,0.5);
          \draw [blue](1.5,1.2) -- (1.5,3.8);
          \draw [blue](2.2,3.5) -- (3.8,3.5);
          \draw [blue](3.5,2.2) -- (3.5,2.8);
          \draw [red](0.5,0.2) -- (0.5,4.8);
          \draw [red](1.2,4.5)--(4.8,4.5);
          \draw [red](4.5,3.8)--(4.5,1.2);
          \draw [red](2.2,1.5)--(3.8,1.5);
          \draw [red](2.5,2.8)--(2.5,2.2);
  \end{scope}
\end{tikzpicture} \begin{tikzpicture}[scale=.5]

  \begin{scope}
    \draw (0, 0) grid (7, 7);

    \setcounter{row7}{1}
    \setrowSEVEN{\circl}{\circl}{\circl }  {\circl }{\circl}{\circl}{\circl}
    \setrowSEVEN{}{}{}{}{}{}{\circl}
    \setrowSEVEN{}{\circl}{\circl}{\circl}{\circl}{}{\circl}
    \setrowSEVEN{}{\circl}{}{}{\circl}{}\circl{}
    \setrowSEVEN{}{\circl}{}{\circl}{\circl}{}{\circl}
    \setrowSEVEN{}{\circl}{}{}{}{}{\circl}
    \setrowSEVEN{}{\circl}{\circl}{\circl}{\circl}{\circl}{\circl}

    \node[anchor=center] at (3.5, -0.5) {$N=7$};
    		   \draw [blue](0.3,6.5) -- (6.8,6.5);
              \draw [blue](6.5,5.8) -- (6.5,0.2);
              \draw [blue](1.2,0.5) -- (5.8,0.5);
              \draw [blue](1.5,1.2) -- (1.5,4.8);
              \draw [blue](2.2,4.5) -- (4.8,4.5);
              \draw [blue](4.5,2.2) -- (4.5,3.8);
              \draw [blue] (3.2,2.5) -- (3.8,2.5);
              \draw [red](0.5,0.2) -- (0.5,5.8);
              \draw [red](1.2,5.5)--(5.8,5.5);
              \draw [red](5.5,4.8)--(5.5,1.2);
              \draw [red](2.2,1.5)--(4.8,1.5);
              \draw [red](2.5,3.8)--(2.5,2.2);
              \draw [red](3.2,3.5)--(3.8,3.5);

  \end{scope}
\end{tikzpicture}\medskip\\

The Pulsar Sequence pieces can be visualized as follows:
\[\underbrace{1},\underbrace{2,1},\underbrace{3,2,1},\underbrace{4,2,3,1},\underbrace{5,2,3,4,1},\underbrace{6,2,4,3,5,1},\underbrace{7,2,5,4,3,6,1},\underbrace{8,2,6,4,5,3,7,1},\ldots\]

We extend this by arranging the Pulsar Sequence pieces in a triangular array, with the $a$-th Pulsar piece placed in row $a$. We also arrange the Dual Sequence pieces in a separate triangle array with the $a$-th Dual piece placed in row $a$.
Then, we label cells in the Pulsar Sequence triangular array in form $p(a,b)$ and we label cells in the Dual Sequence triangular array in form $d_N(a,b)$ where $a$ is the row the cell is in within its appropriate array, $b$ is the column the cell is in within the array, and $N$ is the size of the puzzle such that $1 \le b \le a$ and $1 \le a \le N$. An example for a puzzle of size $N=5$ is shown:
\medskip\\

\noindent
\begin{minipage}[t]{0.45\linewidth} 
\textbf{Pulsar Sequence Array}\\[15pt]
\small 
{\color{red}
\indent\hspace{1mm}$p(1,1)$\\[12pt]
\indent\hspace{1mm}$p(2,1)\; p(2,2)$\\[12pt]
\indent\hspace{1mm}$p(3,1)\; p(3,2)\; p(3,3)$\\[12pt]
\indent\hspace{1mm}$p(4,1)\; p(4,2)\; p(4,3)\; p(4,4)$\\[12pt]
}
\end{minipage}
\hfill
\begin{minipage}[t]{0.45\linewidth} 
\textbf{Dual Sequence Array}\\[15pt]
\small
{\color{blue}
\indent\hspace{1mm}$d_5(1,1)$\\[12pt]
\indent\hspace{1mm}$d_5(2,1)\; d_5(2,2)$\\[12pt]
\indent\hspace{1mm}$d_5(3,1)\; d_5(3,2)\; d_5(3,3)$\\[12pt]
\indent\hspace{1mm}$d_5(4,1)\; d_5(4,2)\; d_5(4,3)\; d_5(4,4)$\\[12pt]
\indent\hspace{1mm}$d_5(5,1)\; d_5(5,2)\; d_5(5,3)\; d_5(5,4)\; d_5(5,5)$\\[12pt]
}
\end{minipage}

\noindent\begin{tikzpicture}[scale=1.0]
  \matrix (M1) [matrix of nodes,
                nodes in empty cells,
                column sep=-\pgflinewidth,
                row sep=-\pgflinewidth,
                nodes={draw, minimum size=1.05cm, anchor=center, inner sep=0pt, font=\fontsize{8.5}{1}\selectfont}] 
  {
    {} & {} & {} & {} & {}\\
    |[text=red]|{$p(4,1)$} & |[text=red]|{$p(3,3)$} & |[text=red]|{$p(3,2)$} & |[text=red]|{$p(3,1)$} & {}\\
    |[text=red]|{$p(4,2)$} & {} & {} & |[text=red]|{$p(2,2)$} & {}\\
    |[text=red]|{$p(4,3)$} & {} & |[text=red]|{$p(1,1)$} & |[text=red]|{$p(2,1)$} & {}\\
    |[text=red]|{$p(4,4)$} & {} & {} & {} & {} \\   
  };

  \newcommand{\MyCircleAt}[2]{%
    \node[draw,circle,minimum size=0.97cm,font=\small] at (#1) {#2};%
  }

  \MyCircleAt{M1-1-1}{}
  \MyCircleAt{M1-1-2}{}
  \MyCircleAt{M1-1-3}{}
  \MyCircleAt{M1-1-4}{}
  \MyCircleAt{M1-1-5}{}
  \MyCircleAt{M1-2-5}{}
  \MyCircleAt{M1-3-5}{}
  \MyCircleAt{M1-4-5}{}
  \MyCircleAt{M1-5-5}{}
  \MyCircleAt{M1-5-2}{}
  \MyCircleAt{M1-5-3}{}
  \MyCircleAt{M1-5-4}{}
  \MyCircleAt{M1-4-2}{}
  \MyCircleAt{M1-3-2}{}
  \MyCircleAt{M1-3-3}{}

  \matrix (M2) [matrix of nodes,
                nodes in empty cells,
                column sep=-\pgflinewidth,
                row sep=-\pgflinewidth,
                nodes={draw, minimum size=1.05cm, anchor=center, inner sep=0pt, font=\fontsize{7.7}{1}\selectfont},
                right=1.15cm of M1]  
  {
    |[text=blue]|{$d_5(5,5)$} & |[text=blue]|{$d_5(5,4)$} & |[text=blue]|{$d_5(5,3)$} & |[text=blue]|{$d_5(5,2)$} & |[text=blue]|{$d_5(5,1)$}\\
    {} & {} & {} & {} & |[text=blue]|{$d_5(4,4)$}\\
    {} & |[text=blue]|{$d_5(2,1)$} & |[text=blue]|{$d_5(1,1)$} & {} & |[text=blue]|{$d_5(4,3)$}\\
    {} & |[text=blue]|{$d_5(2,2)$} & {} & {} & |[text=blue]|{$d_5(4,2)$}\\
    {} & |[text=blue]|{$d_5(3,1)$} & |[text=blue]|{$d_5(3,2)$} & |[text=blue]|{$d_5(3,3)$} & |[text=blue]|{$d_5(4,1)$} \\   
  };

  \MyCircleAt{M2-1-1}{}
  \MyCircleAt{M2-1-2}{}
  \MyCircleAt{M2-1-3}{}
  \MyCircleAt{M2-1-4}{}
  \MyCircleAt{M2-1-5}{}
  \MyCircleAt{M2-2-5}{}
  \MyCircleAt{M2-3-5}{}
  \MyCircleAt{M2-4-5}{}
  \MyCircleAt{M2-5-5}{}
  \MyCircleAt{M2-5-2}{}
  \MyCircleAt{M2-5-3}{}
  \MyCircleAt{M2-5-4}{}
  \MyCircleAt{M2-4-2}{}
  \MyCircleAt{M2-3-2}{}
  \MyCircleAt{M2-3-3}{}
\end{tikzpicture}\medskip\\

The previous paper \cite{pulsar1} introduces the Dual function $(p(a,b))_N$ on ${1, 2, . . . , N}$ where $(p(a,b))_N=N+1-p(a,b)=d_N(a,b)$, such that applying the Dual function twice results in the original input, i.e. $((p(a,b))_N)_N=p(a,b)$. The paper proves that $d_N(a,b) = N+1-p(a,b)$ gives the Dual sequence, which fills in the circled cells of the puzzle. Additionally, it was proven that the $i$'th Pulsar piece contains $p(i,1), p(i,2), \ldots, p(i,i)$ and satisfies the symmetric sum property $p(i,1)+p(i,i)=p(i,2)+p(i,i-1)=\cdots = i+1$.

The paper \cite{pulsar1} concludes with the problem of finding a general formula for the exact terms of $p(a,b)$. In this paper, we solve this problem and find the closed-form expression for $p(a,b)$ in Theorem 5. In combination with knowing $d_N(a,b)$, we are able to immediately solve any Pulsar puzzle of any size $N$.

\section*{New Main Results}

Adapting the proof for Theorem 1 in the previous paper \cite{pulsar1}, we get the following:

\begin{lemma}\label{lem:first-last-gap}
For each row $j$ with $2 \le j \le N$, the entry in the first column of a Pulsar puzzle is exactly one less than the entry in the last column of the same row.
\end{lemma}

\begin{proof}
The first column of the puzzle, minus the cell in the first row, is the $(N-1)$-th piece of the Pulsar sequence, so the corresponding element in $P$ is $p(N-1,j-1)$.  The last column, minus the cell in the first row, is the $N$-dual of the $(N-1)$-th piece of the Pulsar sequence, but reversed. Hence the corresponding element in the puzzle is $N+1-p(N-1,N+1-j)$.  Subtracting, we get $N+1-p(N-1,N+1-j)-p(N-1,j-1)=N+1-N=1$, so in fact each entry in the first column is one smaller than each entry in the last column (in rows $2$ through $N$).
\end{proof}

~\\[-10pt]

We begin by computing the extreme values within each piece of the Pulsar sequence. Recall the Pulsar sequence array.

\begin{center}\begin{minipage}[t]{0.45\linewidth}\textbf{Pulsar Sequence Array}

\begin{tabular}{cccccc}
1\\
2&1\\
3&2&1\\
4&2&3&1\\
5&2&3&4&1\\
6&2&4&3&5&1\\
\vdots
\end{tabular}\end{minipage}\end{center}

The last value of each piece of the Pulsar sequence, appearing as the outer right diagonal of the Pulsar sequence  array, turns out to always be $1$. 

\begin{theorem}\label{thm:rightmost} For each $N$, $p(N,N)=1$.  
\end{theorem}
\begin{proof} Consider a puzzle of size $N+1$.  The bottom left, uncircled, cell is always equal to $1$ because the remaining cells of the bottom row are all circled (and, by Theorem 1, each circled number is greater than each uncircled number).  This cell appears in the Pulsar sequence as $p(N,N)$. \end{proof}



We now turn to the first value in each piece, appearing as the first column of the Pulsar sequence array.  It turns out that the pattern we see continues.

\begin{theorem}  For each $N$, $p(N,1)=N$.
\end{theorem}
\begin{proof}
Applying the symmetric sum property, $p(N,1)+p(N,N)=N+1$.  By Theorem \ref{thm:rightmost}, $p(N,N)=1$, and the result follows.\end{proof}


\begin{theorem}
In rows $3,4,...,N-1$ of the $N$ sized puzzle, values in the first column are exactly 1 greater than values in the $(N-1)$-th column of the same row. Furthermore, for all integers $a \ge 3$ and $2 \le b \le a-1$, $p(a,b) = p(a-2,a-b) + 1$.
\end{theorem}
\begin{proof}

Take $N>3$ and consider its Pulsar puzzle $P$. Removing the first row and the first column gives us an $(N-1)\times (N-1)$ Pulsar puzzle  $P'$, which is itself a Pulsar puzzle rotated $90^\circ$ clockwise. Removing the first 2 rows of $P$ and its first and last columns gives us an $(N-2)\times (N-2)$ Pulsar puzzle  $P'''$, which is itself another Pulsar puzzle rotated $180^\circ$ clockwise. We illustrate $P$, $P'$, $P'''$ below.

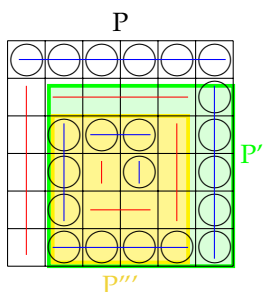
\begin{center}\begin{tikzpicture}[scale=.5]
  \begin{scope}
 \filldraw[fill=green!15, draw=green!100, line width=0.5mm, draw opacity=80](1.1,0) rectangle (6,4.8);
 \filldraw[fill=yellow!55, draw=yellow!100, line width=0.5mm, draw opacity=120](1.2,0.1) rectangle (4.8,4);

        \draw (0, 0) grid (6, 6);
   \node[color=green] at (6.5,3) {P'};
   \node[color=darkyellow] at (3,-0.5) {P'''};
    \setcounter{row6}{1}
    \setrowSIX {\circl }{ \circl}{\circl }{\circl }{\circl }{\circl}
    \setrowSIX {}{}{}{}{}{\circl}
    \setrowSIX {}{\circl}{\circl}{\circl}{}{\circl}
    \setrowSIX {}{\circl}{}{\circl}{}{\circl}
    \setrowSIX {}{\circl}{}{}{}{\circl}
    \setrowSIX {}{ \circl}{\circl }{\circl }{\circl }{\circl}
    \node[anchor=center] at (3, 6.5) {P};
      \draw [blue](0.3,5.5) -- (5.8,5.5);
      \draw [blue](5.5,4.8) -- (5.5,0.2);
      \draw [blue](1.2,0.5) -- (4.8,0.5);
      \draw [blue](1.5,1.2) -- (1.5,3.8);
      \draw [blue](2.2,3.5) -- (3.8,3.5);
      \draw [blue](3.5,2.8) -- (3.5,2.2);
      \draw [red](0.5,0.3) -- (0.5,4.8);
      \draw[red](1.2,4.5)--(4.8,4.5);
      \draw[red](4.5,3.8)--(4.5,1.2);
      \draw[red](2.2,1.5)--(3.8,1.5);
      \draw[red](2.5,2.8)--(2.5,2.2);

  \end{scope}
\end{tikzpicture}\end{center}

By Lemma 1, in rows $3,4,...,N-1$ of P, values in the first column are exactly 1 less than values in the last column of the same row which consists of only circled cells. Values in the circled cells are also all greater than values in the uncircled cells in each row. Therefore, the values of cells in rows $3,4,...,N-1$ that are in the last column must be the smallest circled value in the entirety of their row, and the uncircled values in the first column must be the largest uncircled value in the entirety of their row. \medskip

Using a similar argument in P''', it can be determined that the values of cells in rows $3,4,...,N-1$ that are in the $(N-1)$-th column must be the largest uncircled value in each row of P'''. Therefore, the values of cells in rows $3,4,...,N-1$ that are in the $(N-1)$-th column must be the second-largest uncircled value in the entirety of their row. So this means that in each row $3,4,...,N-1$, the value of the cell in the first column is exactly 1 greater than the value of the cell in the $(N-1)$-th column. \medskip

This can be represented as $p(N-1,b)=p(N-3, N-1-b)+1$ where $2 \le b \le N-2$. To be more general, we have $p(a,b) = p(a-2,a-b) + 1$, where $2 \le b \le a-1$. This means an element in the Pulsar sequence can be calculated from a previous element in the Pulsar sequence. 
\end{proof}


\begin{theorem} \label{thm:closeform}
For all integers $a \ge 1$ and $1 \le b \le a$, the Pulsar sequence element $p(a,b)$ is given by $p(a,b) = (a+1)\cdot\frac{1+(-1)^n}{2}-b\cdot(-1)^n$ where $n=\frac{a-1}{2}-|b-\frac{a+1}{2}|$.
\end{theorem}

\begin{proof}
We proceed by first verifying the boundary cases when $b=a$ or $b=1$.
If $b=a$, then $n=\frac{a-1}{2}-\frac{a-1}{2}=0$ so $p(a,a)=(a+1)-a=1$ which is also what we proved in Theorem 2. If $b=1$, then $n=\frac{a-1}{2}-|1-\frac{a+1}{2}|=\frac{a-1-a-1}{2}+1=0$ so $p(a,1) = (a+1)-1= a$ which is also what we proved in Theorem 3.

Next, we compute the base cases. For $a=1,2,3,4$ we can compute directly from the recursive definition $p(a,b) = p(a-2,a-b) + 1$:

\begin{align*}
p(1,1) &= 1 \\
p(2,1) &= 2, & p(2,2) &= 1 \\
p(3,1) &= 3, & p(3,2) &= 2, & p(3,3) &= 1 \\
p(4,1) &= 4, & p(4,2) &= 2, & p(4,3) &= 3, & p(4,4) &= 1
\end{align*}

\noindent These all agree with the closed form, so the base cases are verified. \smallskip\\

By the induction hypothesis, assume the closed form holds for some fixed $a \ge 5$ and all $2 \le b \le a-1$.\\

\noindent From Theorem 4, we know,
\[
p(a+2,a+2-b)=p(a,b)+1 \quad \text{for} \quad 2 \le (a+2-b) \le (a+1).
\]

By the induction hypothesis, 
\[
p(a,b) = (a+1)\cdot\frac{1+(-1)^n}{2}-b\cdot(-1)^n  \quad \text{, where} \quad n=\frac{a-1}{2}-|b-\frac{a+1}{2}|.
\]
Substituting this into the recurrence gives,
\begin{align*}
p(a+2,a+2-b) &=(a+1)\cdot\frac{1+(-1)^n}{2}-b\cdot(-1)^n+1 \\
             &=\tfrac{a+3}{2}+(a+1)\cdot \frac{(-1)^n}{2} - b\cdot(-1)^n \\
             &=\tfrac{a+3}{2}+(a+1)\cdot \frac{(-1)^n}{2} - (a+2)\cdot(-1)^n + (a+2)\cdot(-1)^n -b\cdot(-1)^n \\
             &=\tfrac{a+3}{2}+(a+1-2a-4)\cdot \frac{(-1)^n}{2}+(a+2-b)\cdot(-1)^n.
\end{align*}
\noindent Then the expression becomes: 
\[
p(\underline{a+2},\uwave{a+2-b})=(\underline{a+2}+1)\cdot\frac{1+(-1)^{n+1}}{2}-\uwave{(a+2-b)}\cdot(-1)^{n+1}.
\]
where $n+1$ can be rewritten as:
\begin{align*}
n+1 &=\tfrac{a-1}{2} - \bigl|\,b - \tfrac{a+1}{2}\,\bigr|+1 \\
    &=\tfrac{(a+2)-1}{2}-\bigl|\,(a+2-b) - \tfrac{(a+2)+1}{2}\,\bigr|.
\end{align*}
This matches exactly the closed form for $p(a+2, a+2-b)$. \medskip\\

The boundary cases $b=1$ and $b=a$ were verified at the start, and the induction step covers $2\le b\le a-1$; hence, by induction on $a$, the formula is valid for all $a\ge 1$ and all $1\le b\le a$. This completes the proof.\end{proof}


\begin{corollary}
An alternate formula for $p(a,b)$, where equivalence is modulo 2: 

\noindent Set $q(a,b)=\begin{cases}1 & 2b< a+1\\ a & 2b \ge a+1\end{cases}$.  Then $p(a,b)=\begin{cases}a+1-b & b\equiv q(a,b)\pmod{2} \\ b & b\not \equiv q(a,b)\pmod{2}\end{cases}$.

\end{corollary}
\begin{proof}
First consider the case when $1\le b\le \frac{a+1}{2}$. From Theorem \ref{thm:closeform}, we have $n=\frac{a-1}{2}-|b-\frac{a+1}{2}|=\frac{a-1}{2}-(\frac{a+1}{2}-b)=b-1$. So when $b\equiv 1, n\equiv 0$, and when $b\equiv 0, n\equiv 1$. Recall the formula from Theorem $5$:
\[p(a,b)=(a+1)\cdot\frac{1+(-1)^n}{2}-b\cdot(-1)^n. \]
So when $b\equiv 1$, $p(a,b) = a+1-b$. When $b\equiv 0$, $p(a,b) = b$.\medskip\\
From the symmetric sum property, we know:\\
$p(a,1)+p(a,a)=p(a,2)+p(a,a-1)=\cdots =p(a,a-b+1)+p(a,b) = a+1$. So for the case when $\frac{a+1}{2} \le b \le a$, we have $p(a,b) = a+1-p(a,a-b+1)$.\medskip\\
When $(a-b+1) \equiv 1$, it means $(a-b) \equiv 0$ or $a\equiv b$. When $(a-b+1) \equiv 0$, it means $(a-b) \equiv 1$ or $a \not\equiv b$.\medskip\\ 
Therefore, when $a\equiv b$: \\
$p(a,b)=a+1-p(a, a-b+1)=a+1-[(a+1)-(a-b+1)]=a+1-b$. \\
Otherwise: \\
$p(a,b)=a+1-p(a, a-b+1)=a+1-(a-b+1)=b$.\end{proof}

\end{document}